\documentclass[12pt,a4paper]{article}
\usepackage{jansstylefile} % The option nobibliography does not print the bibliography. But bibliography can still be used.
\usepackage{enumerate}
\title{Necessary and sufficient conditions for identifiability in the admixture model}
\author{Jan van Waaij\footnote{The author is supported by the Independent Research Fund Denmark (grant number: 8021-00360B) and the University of Copenhagen through the Data+ initiative.}}
\begin{document}

\abstract{
We consider $M$ SNP data from $N$ individuals who are an admixture of $K$ unknown ancient populations. 
Let $\Pi_{si}$ be the frequency of the reference allele of individual $i$ at SNP $s$. So the number of reference alleles at SNP $s$ for a diploid individual is distributed as $\binomial(2,\Pi_{si})$.  
We suppose $\Pi_{si}=\sum_{k=1}^KF_{sk}Q_{ki}$, where $F_{sk}$ is the allele frequency of SNP $s$ in population $k$ and $Q_{ki}$ is the proportion of population $k$ in the ancestry of individual $i$.
 I am interested in the identifiability of $F$ and $Q$, up to a relabelling of the ancient populations. Under what conditions, when $\Pi =F^1Q^1=F^2Q^2$ are $F^1$ and $F^2$ and $Q^1$ and $Q^2$ equal? I show that the anchor condition (\cite{CabrerosStorey2019}) on one matrix together with an independence condition on the other matrix is sufficient for identifiability. 
  I will argue that the proof of the necessary condition in \cite{CabrerosStorey2019} is incorrect, and I will provide a correct proof, which in addition does not require knowledge of the number of ancestral populations. 
  I will also provide abstract necessary and sufficient conditions for identifiability. I will show that one cannot deviate substantially from the anchor condition without losing identifiability. Finally, I show necessary and sufficient conditions for identifiability for the non-admixed case.  }

\section{Introduction}
We consider the following model (see for instance \cite{CabrerosStorey2019,Garcia-ErillAlbrechtsen2020}). There are $N$ individuals, $M$ diallelic sites and $K$ ancestral populations.  Let $Q$ be the $K\x N$ matrix of of admixture proportions, so fraction $Q_{ki}$ of individual $i$'s genome comes from population $k$. Note that the $i$-th column of $Q$, $Q_{\star i}$, is a probability vector, $i\in\set{1,\ldots,N}$. That is, $Q_{ki}\ge 0$ for all $k\in\set{1,\ldots,K}$ and $i\in \set{1,\ldots,N}$ and $\sum_{k=1}^KQ_{ki}=1$, for all $i\in\set{1,\ldots,N}$. Let $F$ be a real $M\x K$ matrix of ancestral SNP frequencies. So $F_{sk}\in[0,1]$, for all $s\in\set{1,\ldots,M}$ and $k\in\set{1,\ldots,K}$. 
So fraction $F_{sk}$ of ancestral population $k$ has the reference allele at locus $s$. Then the expected frequency of the reference allele of SNP $s$ of individual $i$ is \[
\Pi_{si}=(FQ)_{si}=\sum_{k=1}^KF_{sk}Q_{ki}.  
\]

The observed genotype $G_{si}$ of a (diploid) individual is the number (0,1, or 2) of reference variants at SNP $s$ in individual $i$. We assume \begin{equation}\label{eq:admixturemodel}
G_{si}\mid \Pi_{si}\stackrel{\text{independent}}\sim \binomial(2,\Pi_{si}). 
\end{equation}

It follows that $\Pi$ is identifiable, and the law of $G$ only depends on $\Pi$. %An individual $i$ inherits fraction $Q_{ki}$ of its genome from ancient population $k$. It is therefore important that the model is identifiable. 
If there are two pairs of matrices $(F^1,Q^1),(F^2,Q^2)$ so that $\Pi=F^1Q^1=F^2Q^2$, the data cannot decide between $(F^1,Q^1)$ or $(F^2,Q^2)$.  In this case, we cannot say how much individual $i$ inherits from ancient population $k$. Is it $Q_{ki}^1$ or $Q_{ki}^2$?

Of course, identifiability is not the same as consistency (i.e. if one has estimators $\hat F,\hat Q$, do they converge to the true $F$ and $Q$). But consistency requires identifiability. It is therefore important to work with models that are identifiable.

As far as I know \cite{CabrerosStorey2019} is the only study that considers identifiability for the admixture model. They introduce the so-called ``anchor condition" on either $F$ or $Q$, as a sufficient condition for identifiability. They require linearly independent columns for $F$ and linearly independent rows for $Q$, respectively. The anchor condition on $F$ means, that for every  $k\in\set{1,\ldots,K}$ there is a row $s$, so that $F_{s\ell}=0$ when $\ell\neq k$, and $F_{sk}>0$. The anchor condition on $Q$ means, that for every $k\in \set{1,\ldots,K}$ there is a column $i$ so that $Q_{\ell i}=0$ when $\ell\neq k$, and $Q_{ki}=1$, because the columns of $Q$ sum to one. 

So if there is an anchor for $k$ at SNP $s$ (so $F_{sk}>0$ and $F_{s\ell}=0$ when $\ell\neq k$), then if an individual has the reference allele at SNP $s$, then it inherits for sure from population $k$. If an individual $i$ is an anchor for $k$ (so $Q_{ki}=1$ and $Q_{\ell i}=0$ for $\ell\neq k$), then it stores pure information from ancient population $k$.  See also \cite[page 1014]{CabrerosStorey2019}. I refer to \cite[page 2]{Aroraea2013} for a discussion of the anchor condition in topic modelling.

The anchor condition first appeared in \cite{DonohoStodden2004} in the context of non-negative matrix decomposition. Under the anchor condition they provide an algorithm to calculate the matrix decomposition of a non-negative matrix in two non-negative matrices of a given rank. The authors use the term separability condition instead of anchor condition. The anchor condition was first used in the context of topic modelling in \cite{Aroraea2012}, who also introduce this term, and in the context of admixture models by \cite{CabrerosStorey2019}.

 %However, they don't study it in a mathematical rigorous way.
  
I will argue that the  proof of  identifiability with the anchor condition in \cite{CabrerosStorey2019} is not correct.  In this study I treat  identifiability in a mathematical rigorous way. I will give a correct proof for identifiability with the ``anchor" condition (explained below) of \cite{CabrerosStorey2019} and I will discuss abstract necessary and sufficient conditions. In contrast to \cite{CabrerosStorey2019} I allow the number of ancient populations to be unknown.  

\paragraph{Notation}
By $e_i$, we denote the vector with zero entries, except for entry $i$, which is one.
 By $e$, we denote the vector with all entries equal to one. 
 The dimensions of $e_i$ and $e$ are clear from the context. When $A$ is a matrix, we denote by $A_{i\star }$ the $i$-th row of $A$ and by $A_{\star j}$ the $j$-th column of $A$. Note that $A_{i\star}$ is a row vector and $A_{\star j}$ is a column vector. Let $A$ be a real  $n\x m$-matrix. We denote by $\co(A)=\co(\set{A_{\star1},\ldots, A_{\star m}})$ the convex hull generated by the \textbf{columns} of $A$ (see \cref{def:convexhull}) and by $\cone(A)=\cone(\set{A_{1\star},\ldots,A_{n\star}})$ the cone generated by the \textbf{rows} of $A$ (see \cref{def:coneandwedge,def:wedgeorconegeneratedbysetA}, assuming that it is a cone and not just a wedge). 
\paragraph{}

\section{Identifiability}

Because $\Pi$ is identifiable from the law of $G$, it is sufficient to show that when $\Pi=F^1Q^1=F^2Q^2$, that $F^1$ and $F^2$ and $Q^1$ and $Q^2$ are equal, where $(F^1,Q^1),(F^2,Q^2)$ are pairs of matrices in our model.
As for any permutation $\pi:\set{1,\ldots,K}\to \set{1,\ldots,K}$, 
\[
\Pi_{si}=\sum_{k=1}^KF_{s\pi(k)}Q_{\pi(k)i},
\]
the best we can hope for is that $F^1$ and $F^2$ and $Q^2$ and $Q^2$ are equal up to a permutation of the columns of $F^2$ and the rows of $Q^2$. 
The permutation $\pi$ corresponds to a relabelling of the ancient populations.

Let $K\in\NN$. 
Let $\sF_K$ be the set of real $M\x K$ matrices $F$ with $0\le F_{sk}\le 1$ for all $s\in\set{1,\ldots,M}, k\in\set{1,\ldots,K}$. 
The matrices $F\in\sF_K$ represent the allele frequencies of the ancient populations. 
Let $\sQ_K$ be the set of real $K\x N$ matrices so that $0\le Q_{ki}\le 1$ for all $k\in\set{1,\ldots,K}$ and $i\in\set{1,\ldots,N}$ and the columns of $Q$ sum to one, that is, 
for all $i\in\set{1,\ldots,N}$, $\sum_{k=1}^KQ_{ki}=1$. 

We define an equivalence relation on the pairs $(F,Q)\in \sum_{k=1}^\infty \sF_K\x \sQ_K$. We say that $(F^1,Q^1)$ is equivalent to $(F^2,Q^2)$ when $F_1$ and $F_2$ have the same number of columns $K$ (which is then the number of rows of $Q^1$ and $Q^2$) and there is a permutation $\pi:\set{1,\ldots,K}\to\set{1,\ldots,K}$ of the columns of $F^1$ and the rows of $Q^1$, 
so that $F^2_{sk}=F^1_{s\pi(k)}$ and $Q^2_{ki}=Q^1_{\pi(k)i}$. Notation $(F^1,Q^1)\sim (F^2,Q^2)$. In this case $F^1Q^1=F^2Q^2$. 
When $(F^1,Q^1)$ and $(F^2,Q^2)$ are not equivalent, we write $(F^1,Q^1)\not\sim (F^2,Q^2)$. 

\begin{definition}
    A model subset $\sM\subseteq \bigcup_{K=1}^\infty \sF_K\x \sQ_K$ is identifiable for the admixture model \cref{eq:admixturemodel} if for every $(F^1,Q^1),(F^2,Q^2)\in \sM$, 
    $F^1Q^1=F^2Q^2$ implies that $(F^1,Q^1)\sim (F^2,Q^2)$. 
\end{definition}

%When $K$ is known, \cite{CabrerosStorey2019} propose the following criteria for identifiability: either an anchor condition on $Q$ or an anchor condition on $F$ (explained below). 
%Very confusingly, only in the proof, in the appendix, they require that the other matrix has independent columns or rows, respectively. 
%This is absolutely necessary, as otherwise, the statement is not valid. Although the independence condition on $F$ can be slightly relaxed as we will below. 
 I will now specify the anchor and independence conditions what are sufficient for identifiability. 
 
 Let $\sQ_K^{\text{an}}$ be the subset of $\sQ_K$ of all matrices $Q$ so that for every $k\in\set{1,\ldots, K}$ there is an $i\in\set{1,\ldots,N}$, so that $Q_{ki}=1$ and $Q_{\ell i}=0$ for all $\ell\neq k$. Note that $\sQ_K^{\text{an}}$ is empty when $K>N$.  We say that matrices $Q\in \sQ_K^{\text{an}}$ satisfy the anchor condition on $Q$. The individual $i$ so that  $Q_{i\star}=e_k$, for some $k$, is called an anchor individual, or an unadmixted individual.   
Let $\sF_K^{\text{an}}\subseteq \sF_K $ be the set of matrices so that for each $k\in\set{1,\ldots,K}$, there is an $s\in\set{1,\ldots,M}$ so that for all $\ell\neq k$, $F_{s,\ell}=0$ and $F_{s,k}>0$. Note that $\sF_K^{\text{an}}$ is empty when $K>M$. We say that matrices $F\in \sF_K^{\text{an}}$ satisfy the anchor condition on $F$. A SNP $s$ so that $F_{s\star} = \delta e_k'$, for some $\delta>0$, is called an anchor SNP. %So an individual with reference allele at SNP $s$, has for sure ancestry from population $k$. (See \cite[page 1014]{CabrerosStorey2019} in the admixture model and \cite[page 2]{Aroraea2013} in topic modelling.)

%In the following theorem, I provide a correct proof. In addition, my theorem doesn't require the knowledge of $K$. 

Let us determine independence conditions on the matrices. Let $\sF_K^{\text{in}}$ be the subset of $\sF_K$ of all matrices $F$ so that  $F_{\star,1}-F_{\star,K},\ldots,F_{\star,K-1}-F_{\star,K}$ are linearly independent vectors. Note that $\sF_K^{\text{in}}$ is empty when $M<K-1$. 
    Let $\sQ_K^{\text{in}}$ be the subset of $\sQ_K$ of all matrices $Q\in\sQ_K$ so that the rows of $Q$ are linearly independent. Note that $\sQ_K^{\text{in}}$ is empty when $K>N$.  
    The independence condition on the columns of $F$ is slightly lighter than the independence condition on the rows of $Q$. This is because we use that the other matrix in the pair has rows that sum to one. This basically reduces the dimension by one.

    \subsection{The argument of \cite{CabrerosStorey2019}}

Let $\tilde \sF_K^{\text{in}}$ be all $F\in \sF_K$ so that all columns of $F$ are linearly independent (i.e. all rank $K$ matrices in $\sF_K$). Note that $\tilde \sF_K^{\text{in}}\subseteq \sF_K^{\text{in}}$. 
\cite{CabrerosStorey2019} assure that for given (known) $K$, $\sF_K^{\text{an}}\x \sQ_K^{\text{id}}$ and $\tilde \sF_K^{\text{id}}\x \sQ_K^{\text{an}}$ are identifiable models.  They ``proof" identifiability for the anchor condition on $F$, and state that the proof with the anchor condition on $Q$ is similar. 
  
The ``proof" of \cite{CabrerosStorey2019} is as follows. First they suppose that for $(F,Q)$ in the model, when $\Pi = FQ$  you may assume that $F$ has the form \begin{equation}\label{eq:blockformofF}
    F = \begin{pmatrix}
        D \\ A
    \end{pmatrix},
\end{equation}
where $D$ is a diagonal matrix with positive entries and $A$ are the $M-K$ bottom rows of $F$. Next they argue that when $\Pi'$ is the submatrix of $\Pi$ formed from the first $K$ rows of $\Pi$, then $Q=D^{-1}\Pi'$. Finally, the fact that the rows of $Q$ are linearly independent, uniquely identifies $A$. 

Although this argument sounds convincing, it is wrong. We have to show that when $(F^1,Q^1),(F^2,Q^2)$ are in the model, so that $\Pi=F^1Q^1=F^2Q^2$ that $(F^1,Q^1)\sim (F^2,Q^2)$. To write $F^1$ in the form \cref{eq:blockformofF}, one has to permute the rows of $F^1$, and hence also the rows of $\Pi$. However, if after a permutation of the rows of $F^1$, $F^1$ is of the form \cref{eq:blockformofF}, then under the same permutation $F^2$ is not necessarily of this form, let alone that $D^1=D^2$ (where $D^1$ and $D^2$ are defined similarly as $D$). 
 At the very least, this should be proven. 
 Doing a different permutation of the rows of $F^1$ and $F^2$ would also not work, as after the transformation, the resulting $\Pi$'s are not necessarily equal, which is crucial in the next step of their proof. So \cite{CabrerosStorey2019} do not establish identifiability. % that also means that this proof cannot uniquely identify $Q$.   
%Next, they say that $Q= D^{-1}\Pi'$, where $\Pi'$ is the submatrix formed from the first $K$ rows of $\Pi$, uniquely determines $Q$. However, $D$ was obtained incorrectly, so this does not make sense anymore.  They conclude by arguing that the independent rows of $Q$ determine $A$ uniquely. 

Although $\sF_K^{\text{an}}\x \sQ_K^{\text{id}}$ and $\tilde \sF_K^{\text{id}}\x \sQ_K^{\text{an}}$ are identifiable models (according to \cref{thm:identifiabilityanchorinF,thm:identifiabilityanchorinQ} below), \citeauthor{CabrerosStorey2019}'s argument for identifiability is not correct. 
    
    \subsection{Sufficient conditions for identifiability}
    
    In this section I give a correct proofs for identifiability with an anchor condition on one matrix and an independence matrix on the other matrix. I will allow for a slightly weaker condition on the independence of the columns of $F$ compared with \cite{CabrerosStorey2019}. Futhermore, I don't require knowledge of $K$.  In \cref{thm:identifiabilityanchorinF,thm:identifiabilityanchorinQ} I show that $\bigcup_{K=1}^{\infty }\sF_K^{\text{in}}\x \sQ_K^{\text{an}} = \bigcup_{K=1}^{(M+1)\wedge N}\sF_K^{\text{in}}\x \sQ_K^{\text{an}}$ and $\bigcup_{K=1}^{\infty}\sF_K^{\text{an}}\x \sQ_K^{\text{in}}=\bigcup_{K=1}^{M\wedge N}\sF_K^{\text{an}}\x \sQ_K^{\text{in}}$ are identifiable models. In \cref{thm:identifiabilityanchorinFcounterexample,thm:identifiabilityanchorinQcounterexample} I show that the independence requirements are necessary. In the same theorems I show that we cannot deviate from the anchor condition too much without loosing identifiability. Necessary and sufficient conditions are discussed in \cref{sec:discussion}.

Although the proofs of \cref{thm:identifiabilityanchorinQ,thm:identifiabilityanchorinF} have some similarities, we use the theory of convex sets in the first, and the theory of cones in the second theorem. A cone is a subset $K$ of a real vector space, so that for all $x,y\in K$ and $\lambda\ge 0$, $\lambda x$ and $x+y$ are also in $K$. Additionally, when $-K=\set{-x:x\in K}$,  $K\cap(-K)=\set 0$.

\begin{theorem}\label{thm:identifiabilityanchorinQ}
   Define $\sM'=\bigcup_{K=1}^{(M+1)\wedge N}\sF_K^{\text{in}}\x \sQ_K^{\text{an}}$. 
    Then $\sM'$ is an identifiable model.  
    \end{theorem} 
\begin{proof}
The reader may familiarise him- or herself with the theory of convex sets in \cref{sec:convexsets}. 
For an $m\x n$-matrix $A$, we denote by $\co(A)=\co(A_{\star 1},\ldots,A_{\star n})$ the convex hull spanned by the \textbf{columns} of $A$.  

Let $(F^1,Q^1),(F^2,Q^2)\in \sM'$ and suppose that $\Pi=F^1Q^1=F^2Q^2$. Let $K_1$ the number of columns of $F_1$, and $K_2$ the number of columns of $F_2$.  
Note that each column of $\Pi$ is a convex combination of the $K$ columns of $F^1$ (or of $F^2$). It follows that $C:=\co(\Pi)\subseteq \co(F^1)$ and $C\subseteq \co(F^2)$. 

 As $e_1,\ldots,e_K$ are columns in $Q$, it follows that $C $ contains the vectors that generate $\co(F^1)$ and $\co(F^2)$, so $C=\co(F^1)=\co(F^2)$.  
By our assumption, for each $(F,Q)\in \sM'$,  $F_{\star 1}-F_{\star K},\ldots,F_{\star K-1}-F_{\star K}$ are linearly independent. So by \cref{rem:uniquedecompositionbygeneratorsimpliesthatthegeneratorsareextreme}  $\set{F_{\star 1}^1,\ldots,F_{\star K_1}^1}$ and $\set{F_{\star 1}^2,\ldots,F_{\star K_2}^2}$ are two sets of extreme points (see \cref{def:extremeset}) that  generate $C$. It follows from \cref{cor:atmostoneminimalset} that the two sets are equal, in particular $K_1=K_2$. So there is a permutation $\pi:\set{1,\ldots,K_1}\to \set{1,\ldots,K_1}$ so that $F_{\star k}^2=F_{\star \pi(k)}^1$, for all $k\in\set{1,\ldots,K_1}$. 
It follows from \cref{lem:uniqueconvexcombinationintermsofindependentvectors} and the fact that $F_{\star 1}^1-F_{\star K_1}^1,\ldots,F_{\star K_1-1}^1-F_{\star K_1}^1$ are linearly independent that each element in $C$ has a unique convex decomposition in terms of $F_{\star 1}^1,\ldots,F_{\star K_1}^1$. In particular, $Q^2_{k\star }=Q^1_{\pi(k)\star }$, for all $k\in\set{1,\ldots,K_1}$. Thus $(F^1,Q^1)\sim (F^2,Q^2)$ and $\sM'$ is identifiable.
\end{proof}

It turns out that $\sM'$ cannot be substantially enlarged to a model  that is still identifiable:
\begin{theorem}\label{thm:identifiabilityanchorinQcounterexample}
    Let $K\ge 2$ and $N\ge K+1$  and let $F\in \sF_K\weg\sF_K^{\text{in}}$. If $Q\in\sQ_K^{\text{an}}$ contains a column $Q_{\star i}$ so that $Q_{k,i}>0$, for all $k\in\set{1,\ldots,K}$, then there is a $Q^2\in\sQ_K^{\text{an}}$,  so that $FQ=FQ^2$, but $(F,Q)\not\sim (F,Q^2)$. 
    
 Let $K\ge2$ and let $F\in\sF_K^{\text{in}}$ and there is a column $F_{\star k_0}$ and a $0<\delta<1/2$ so that $\delta\le F_{s,k_0}\le 1-\delta$ for all $s\in\set{1,\ldots,M}$  and $Q\in\sQ_K$, then there is a $Q^2\in \sQ_K\weg\sQ_K^{\text{an}}$ and $F^2\in\sF_K^{\text{in}}$, so that $FQ=F^2Q^2$, but $(F,Q)\not\sim (F^2,Q^2)$. The matrix $Q^2$ can be chosen so that $\set{k: e_k \text{ is a column in }Q^2}=\set{k: e_k \text{ is a column in }Q}\weg\set{k_0}$. 
\end{theorem}
\begin{proof}
Let $F\in\sF_K\weg\sF_K^{\text{in}}$ and $Q\in\sQ_K^{\text{an}}$ contains a column $Q_{\star i}$ so that $Q_{k,i}>0$ for all $k\in\set{1,\ldots,K}$. Note that the $i$-th column of $FQ$ is an open convex combination (see \cref{def:openconvexcombination}) of $F_{\star 1},\ldots, F_{\star K}$. By \cref{lem:uniqueconvexcombinationintermsofindependentvectors} in combination with \cref{lem:propertiesopenconvexset} there is a convex combination $c$ of $F_{\star 1},\ldots,F_{\star K}$ different from $Q_{\star  i}$ that results in the $i$-th column of $FQ$. Define $Q^2$ by replacing the $i$-th column of $Q$ by $c$. Then $FQ=FQ^2$ and as the $i$-th column of $Q$ is not equal to one of $e_1,\ldots,e_k$, $Q^2$ is still an element of $\sQ_K^{\text{an}}$. But $(F,Q)\not\sim (F,Q^2)$.

Now consider $F\in\sF_K^{\text{in}}$ and $Q\in\sQ_K$, and assume that there is a column $F_{\star k_0}$ so that $\delta\le F_{s,k_0}\le 1-\delta$ for all $s\in\set{1,\ldots,M}$. After relabelling if necessary, we may assume that $k_0=2$. 
Define the real $K\x K$-matrix $R$ as 
\begin{align*}
    R = \begin{pmatrix}
        R^1 & 0 \\ 0 & I
    \end{pmatrix},
    \intertext{where}
    R^ 1 = & \begin{pmatrix}
        1 & \delta \\ 0 & 1-\delta 
    \end{pmatrix}.
    \intertext{Then $R$ is invertible with inverse}
    R^{-1} = & \begin{pmatrix}
         (R^1)^{-1} & 0 \\ 0 & I 
    \end{pmatrix},
    \intertext{where}
    (R^1)^{-1} = & \begin{pmatrix}
        1 & -\frac{\delta}{1-\delta} \\ 0 & \frac{1}{1-\delta}
    \end{pmatrix}
\end{align*}
and $I$ is the $(K-2)\x (K-2)$ identity matrix. 
Note that $e'RQ=e'Q=e'$. Moreover the entries of $RQ$ are nonnegative. Hence $RQ\in\sQ_K$, and  for all $k\in \set{1,3,\ldots,K}$ and $i\in\set{1,\ldots,M}$, $(RQ)_{ki}\ge Q_{ki}$. So if $Q_{\star i} =e_k$, then $(RQ)_{k i}=1$. If we make use of the fact that the columns of $RQ$ sum to one and are nonnegative, it follows that $(RQ)_{\star i} =e_k$. 
Note that $(RQ)_{2i}=(1-\delta)Q_{2i}\le 1-\delta<1$, hence $e_2$ is \textbf{not} a column of $RQ$.    
 It follows that $\set{k: e_k \text{ is a column in }RQ}=\set{k: e_k \text{ is a column in }Q}\weg\set{2}$.

Note that $(FR^{-1})_{\star k}=F_{\star k}$ for $k\in\set{1,3,\ldots,K}$ and $(FR^{-1})_{2\star } = -\frac\delta{1-\delta} F_{\star1} + \frac1{ 1-\delta}F_{\star2}$, so in case $K=2$, then $(FR^{-1})_{\star 1}-(FR^{-1})_{\star 2}= -\frac1{1-\delta}(F_{\star1}-F_{\star2})$. In case $K>2$, then $(FR^{-1})_{\star k}-(FR^{-1})_{\star K}=F_{\star k}-F_{\star K}$, for $k\in\set{1,3,\ldots,K-1}$ and $(FR^{-1})_{\star 2}-(FR^{-1})_{\star K} = -\frac\delta{1-\delta} F_{\star1} + \frac1{ 1-\delta}F_{\star2}-F_{\star K}=-\frac\delta{1-\delta}( F_{\star1}-F_{\star K}) + \frac1{ 1-\delta}(F_{\star2}-F_{\star K}) $. 
 It follows that as $F_{\star1}-F_{\star K},\ldots,F_{\star K-1}-F_{\star K}$ are linearly independent, also $(FR^{-1})_{\star 1}-(FR^{-1})_{\star K},\ldots,(FR^{-1})_{\star K-1}-(FR^{-1})_{\star K}$ are linearly independent. 
 Note that the first, third, up to the $K$th column of $FR^{-1}$ are identical to $F$. We only need to show that all entries in the second column of $FR^{-1}$ take values in $[0,1]$. 
 Using the assumptions on $F_{\star 2}$ (remember that $k_0=2$), we see that for every $s\in\set{1,\ldots,M}$,
 \begin{align*}
     (FR^{-1})_{s2} = &-\frac{\delta}{1-\delta}F_{s1}+ \frac{1}{1-\delta}F_{s2} .
     \intertext{So}
     (FR^{-1})_{s2} \le &\frac{1}{1-\delta}(1-\delta) =1 
     \intertext{and}
     (FR^{-1})_{s2} \ge & - \frac{\delta}{1-\delta}+ \frac{1}{1-\delta  }\delta =0. 
 \end{align*}
 Hence $FR^{-1}\in \sF_K^{\text{in}}$.  
 Clearly $FQ=(FR^{-1})(RQ)$, but $(F,Q)\not\sim (FR^{-1},RQ)$. 
\end{proof}

\begin{theorem}\label{thm:identifiabilityanchorinF}
 Define $\sM''=\bigcup_{K=1}^{M\wedge N}\sF_K^{\text{an}}\x \sQ_K^{\text{in}}$. 
    Then $\sM''$ is an identifiable model.  
            \end{theorem} 
\begin{proof}
The reader may familiarise him- or herself with the theory of cones in \cref{sec:cones}. 

For an $a\x b$-matrix $A$, with nonnegative entries, we denote by $\cone(A)=\cone(A_{1\star },\ldots,A_{a \star })$ the cone generated by the \textbf{rows} of $A$ (see \cref{def:wedgeorconegeneratedbysetA}). As $A$ has nonnegative elements, this is indeed a cone and not just a wedge (see \cref{def:coneandwedge}). 

Let $(F^1,Q^1),(F^2,Q^2)\in \sM''$ and suppose that $\Pi=F^1Q^1=F^2Q^2$. Let $K_1$ be the number of columns of $F^1$ and $K_2$ the number of columns of $F^2$.  
Note that each row of $\Pi$ is is an element of $\cone(Q^1)$ and an element of $\cone(Q^2)$. As there are  $\delta_1^1,\ldots,\delta_{K_1}^1,\delta_1^2,\ldots,\delta_{K_2}^2>0$, so that $\delta_1^1 e_1',\ldots,\delta_{K_1}^1e_{K_1}'$ are rows in $F^1$, and $\delta_1^2e_1',\ldots, \delta_{K_2}^2e_{K_2}'$ are rows of  $F^2$, it follows that $\cone(\Pi)= \cone(Q^1)=\cone(Q^2)$. 

As $Q_{1\star},\ldots,Q_{K\star}$ are linearly independent for each $Q\in\sQ_K^{\text{in}}$, it follows from \cref{lem:aconegeneratedbyAwithuniquedecompositionthenAisminimalandconsistofextremepoints,lem:coneuniquedecompositionifandonlyifelementsarelinearlyindependent} that the rows of $Q^1$ are extreme points of $\cone(Q^1)$. Similarly, the rows of $Q^2$ are also extreme points of $\cone(Q^1)$. It follows from \cref{lem:uniquenessofminimalsetthatgeneratescone} $K_1=K_2$ and there  is a permutation $\pi:\set{1,\ldots,K_1}\to \set{1,\ldots,K_1}$, and there are constants $\eps_1,\ldots,\eps_{K_1}>0$ so that $Q_{k\star}^2= \eps_k  Q_{\pi(k)\star}^1$, for every $k\in\set{1,\ldots,K_1}$. Define $\eps$ as $\eps'=(\eps_{\pi^{-1}(1)},\ldots,\eps_{\pi^{-1}(K_1)})$, where $\pi^{-1}$ is the inverse mapping of $\pi$. Then $e'Q^2 = \eps'Q^1=e'Q^1=e'$, and as the rows of $Q^1$ are linearly independent, it follows that $\eps=e$. In particular $Q_{k\star}^2=  Q_{\pi(k)\star}^1$, for every $k\in\set{1,\ldots,K_1}$.

As the rows of $Q^1$ (and $Q^2$) are linearly independent, it follows from \cref{lem:coneuniquedecompositionifandonlyifelementsarelinearlyindependent} that every element in $\cone(Q^1)$ has a unique decomposition in terms of the rows of $Q^1$. So $F^2_{\star k}=F^1_{\star \pi(k)}$ for all $k\in\set{1,\ldots,K_1}$. Hence $(F^1,Q^1)\sim (F^2,Q^2)$ and $\sM''$ is an identifiable model.  
\end{proof}

Like $\sM'$, $\sM''$ cannot be substantially enlarged while maintaining identifiability. 
\begin{theorem}\label{thm:identifiabilityanchorinFcounterexample}
 Let $K\ge 2$ and $M\ge K+1$. 
    Let $F\in \sF_K^{\text{an}}$ be so that there is a $\delta>0$ and a row $i$ so  that $\delta\le F_{i,k}\le 1-\delta$, for all $k\in\set{1,\ldots, K}$, and let $Q\in\sQ_K\weg\sQ_K^{\text{in}}$,  then there is an $F^2\in\sF_K^{\text{an}}$, so that $FQ=F^2Q$, but $(F,Q)\not\sim (F^2,Q)$. 

Let $F\in\sF_K$ and $Q\in\sQ_K^{\text{in}}$ be such that there is a  $0<\delta<1/2$ and a row $k_0$ so that $ Q_{k_0i}\ge \delta$ for all $i\in\set{1,\ldots,N}$.
     Then there are $F^2\in \sF_K\weg\sF_K^{\text{an}}$ and $Q^2\in \sQ_K^{\text{in}}$ so that $FQ=F^2Q^2$, but $(F,Q)\not\sim (F^2,Q^2)$. The matrix $F^2$ can be chosen so that $\set{k:\delta_ke_k'\text{ is a row in }F^2\text{ for some }\delta_k>0}=\set{k:\delta_ke_k'\text{ is a row in }F^1\text{ for some }\delta_k>0}\weg\set {k_0}$. 
\end{theorem}
\begin{proof}
    Let $Q\in\sQ_K\weg\sQ_K^{\text{in}}$, so the rows of $Q$ are not 
independent and let 
$F\in\sF_K^{\text{an}}$ so that  for some row $i$ and some  $
\delta>0$, $\delta \le F_{i,k}\le 1-\delta$ 
for all $k\in\set{1,\ldots,K}$. As $Q$ is not 
independent, there is a nonzero vector 
$v\in\re^K$ so that $v'Q=0$. For small 
enough $\alpha>0$, we have that $-\delta 
\le \alpha v_k \le \delta$, for all $k\in\set{1,\ldots,K}$. Let $w'= 
\alpha v'+F_{i\star }$.  So $0\le w_k\le 1$, for 
all $k\in\set{1,\ldots,K}$. Define $F^2$ by 
replacing the $i$-th row by $w'$. As the $i$-th row of $F$ is not equal to $\delta e_k'$ for all $\delta>0$ and $k\in\set{1,\ldots,K}$, we have that $F^2\in \sF_K^{\text{an}}$ and 
$FQ=F^2Q$. But $(F,Q)\not\sim (F^2,Q)$. 

Let $F\in\sF_K$ and $Q\in\sQ_K^{\text{in}}$ be such that there is a  $0<\delta<1/2$ and a row $k_0$ so that $\delta\le Q_{k_0i}\le 1-\delta$ for all $i\in\set{1,\ldots,N}$. After a permutation of the columns of $F$ and the rows of $Q$, if necessary, we may assume that $k_0=2$. 
Define the real $K\x K$-matrix $R$ as follows:
\begin{align*}
    R = &\begin{pmatrix}
        R^1 & 0 \\ 0 & I
    \end{pmatrix},
    \intertext{where}
    R^1 =& \begin{pmatrix}1-\delta & 0 \\ \delta & 1 \end{pmatrix}.
    \intertext{Then $R$ is invertible, with inverse }
    R^{-1} =& \begin{pmatrix}
        (R^1)^{-1} & 0 \\ 0 & I 
    \end{pmatrix},
    \intertext{where}
    (R^1)^{-1} = & \begin{pmatrix}
        \frac1{1-\delta} & 0 \\ \frac{-\delta}{1-\delta} & 1 
    \end{pmatrix}.
\end{align*}
Note that columns $k=2,\ldots,K$ of $FR$ are identical to those of $F$, and the first column of $FR$ is a convex combination of the first two columns of $F$. It follows that $\set{k:\delta_ke_k'\text{ is a row in }F^2\text{ for some }\delta_k>0}=\set{k:\delta_ke_k'\text{ is a row in }F^1\text{ for some }\delta_k>0}\weg\set {2}$.

For $k=3,\ldots, K$, the $k$-th row of $R^{-1}Q$ and $Q$ are identical. Note that $e'R^{-1}Q=e'Q=e'$. So the columns of $R^{-1}Q$ still sum to one. As $R$ is invertible, the rows of $R^{-1}Q$ are also independent. It is only left to show that the entries of the first two rows of $R^{-1}Q$ stay non-negative. This is clear for the first row. 
All entries in the second row of $Q$  are at least $\delta$ (remember that $k_0=2$), all entries in the first row of $Q$  are at most $1-\delta$, as the columns are non-negative and sum to one. Using this, we have for the  second row,
\begin{align*}
        (R^{-1}Q)_{2i} = & -\frac{\delta }{1-\delta}Q_{1i} + Q_{2i}.
        \intertext{So} 
        (R^{-1}Q)_{2i} \le & Q_{2i} \le 1,
        \intertext{and}
        (R^{-1}Q)_{2i} \ge & -\frac{\delta}{1-\delta}(1-\delta) + \delta=0.  
\end{align*}
It follows that $R^{-1}Q\in\sQ_K^{\text{in}}$. Clearly $FQ=(FR)(R^{-1}Q)$, but $(F,Q)\not\sim (FR,R^{-1}Q)$. 
\end{proof}

\section{Necessary conditions}\label{sec:discussion}

So, we  found two different models $\sM'$ and $\sM''$ ($\sM'$ is not contained in $\sM''$, nor vice versa), which each provide identifiability, and both cannot be substantially enlarged without violating the identifiability property. It is remarkable, that $\sM'$ has an independence requirement on the matrices $F$ and an anchor requirement on $Q$, while $\sM''$ has an anchor requirement on $F$ and an independence requirement on $Q$. Note furthermore, that $\sF_K^{\text{an}}\subseteq \sF_K^{\text{in}}$ and $\sQ_K^{\text{an}}\subseteq \sQ_K^{\text{in}}$, so provable identifiability is only maintainable when enlarging $\sF_K^{\text{an}}$ is paired with shrinking $\sQ_K^{\text{in}}$, and vice versa.  
 %\Cref{thm:identifiabilityanchorinQcounterexample,thm:identifiabilityanchorinFcounterexample} show us that $\sM'$ and $\sM''$ cannot be enlarged significantly, without violating the identifiability property.  

The first part of  \cref{thm:identifiabilityanchorinQcounterexample} shows that for any  set $\sF_K'$ strictly larger than $\sF_K^{\text{in}}$, $\set{(F,Q):F\in\sF_K',Q\in \sQ_K^{\text{an}}}$ is not identifiable anymore. 
Similarly, the first part of theorem \cref{thm:identifiabilityanchorinFcounterexample} shows that for any set $\sQ_K'$ strictly larger than $\sQ_K^{\text{in}}$, $\set{(F,Q):F\in\sF_K^{\text{an}},Q\in\sQ_K'}$ is not identifiable anymore. 
At the same time, the second part of \cref{thm:identifiabilityanchorinQcounterexample} shows that there is not much space to enlarge $\sQ_K^{\text{an}}$ in $\sM'$, while maintaining identifiability,   
and similar for $\sM''$ there is not much space to enlarge $\sF_K^{\text{an}}$. 

However, precise practical necessary conditions are still lacking. We will discuss abstract necessary and sufficient conditions in the next subsection. Hopefully they give a direction for future research. 

\subsection{Discussion of necessary conditions}

Let $K(F)$ denote the number of columns of a matrix $F$. Recall that for a real  $n\x m$-matrix $A$, $\co(A)=\co(\set{A_{\star1},\ldots, A_{\star m}})$ and $\cone(A)=\cone(\set{A_{1\star},\ldots,A_{n\star}})$.
We continue with a discussion of necessary conditions for identifiability. Now suppose $\sM$ is an identifiable model. Then for $(F^1,Q^1),(F^2,Q^2)\in \sM$ satisfying $F^1Q^1=F^2Q^2$, we have that the columns of $F^1$ and $F^2$ are equal up to a permutation, and the rows of $Q^1$ and $Q^2$ are also equal up to a permutation. In particular $\co(F^1)=\co(F^2)$ and $\cone(Q^1)=\cone(Q^2)$.  So 
\begin{theorem}\label{thm:identfiabilitynecessaryconditions}
    Let $\sM$ be an identifiable model, then $\co(F^1)=\co(F^2)$ and $\cone(Q^1)=\cone(Q^2)$.
\end{theorem}
In the following theorem, we establish necessary and sufficient conditions for identifiability. 
\begin{theorem}
    Let $\sM$ be a model that satisfies the following properties: 
    \begin{enumerate}[(1)]
    \item $K(F)<N$ for all $(F,Q)\in \sM$, 
    \item For every $(F,Q)\in \sM$, $\set F\x \sQ_{K(F)}^{\text{an}}\subseteq \sM$.     
    \end{enumerate}
    Then $\sM$ is identifiable if and only if 
    \begin{enumerate}[(a)]
        \item for all $(F^1,Q^1),(F^2,Q^2)\in \sM$, $F^1Q^1=F^2Q^2$ implies that $\co(F^1)=\co(F^2)$, 
        \item for every $(F,Q)\in \sM$, $F_{\star 1}-F_{\star K},\ldots,F_{\star K-1}-F_{\star K}$ are linearly independent. 
    \end{enumerate}
\end{theorem}
\begin{proof}
First suppose that  $\sM$ is  identifiable. According to \cref{thm:identfiabilitynecessaryconditions}, (a) holds. Let $(F,Q)\in\sM$. %According to \cref{lem:uniqueconvexcombinationintermsofindependentvectors} every element of $\co(F)$ has a unique representation in terms of the columns of $F$ if and only if $F_{\star1}-F_{\star K},\ldots,F_{\star K-1}-F_{\star K}$ are linearly independent. 
Suppose that $F_{\star1}-F_{\star K},\ldots,F_{\star K-1}-F_{\star K}$ are not linearly independent. Then by \cref{lem:uniqueconvexcombinationintermsofindependentvectors} there are two probability vectors $p,q$, $p\neq q$ so that $p_1F_{\star 1}+ \ldots+p_KF_{\star K}=q_1F_{\star 1}+ \ldots+q_KF_{\star K}$.
Define the $K(F)\x N$ matrices $Q^p$ and $Q^q$ as follows: 
\[
Q^p = (p, I_{K(F)}, e_1,\ldots,e_1), \quad Q^q= (q, I_{K(F)}, e_1,\ldots,e_1),
\]
where $I_{K(F)}$ is the $K(F)\x K(F)$ identity matrix, $Q^p$ and $Q^q$ have $N-K(F)-1\ge 0$ columns $e_1$ in the right side of the matrix. By property (1) $Q^p,Q^q\in \sQ_{K(F)}^{\text{an}}$, and by property (2), $(F,Q^p),(F,Q^q)\in \sM$ and $FQ^p=FQ^q$. But $(F,Q^p)\not \sim (F,Q^q)$. Contradiction. So for every $(F,Q)\in\sM$, $F_{\star1}-F_{\star K},\ldots, F_{\star K-1}-F_{\star K}$ are linearly independent. 

Now assume $\sM$ satisfies (a) and (b). Let $(F^1,Q^1),(F^2,Q^2)\in \sM$ be such that $\Pi= F^1Q^1=F^2Q^2$. It follows from (a) that $\co(F^1)=\co(F^2)$. Let $K_1$ the number of columns of $F^1$ and $K_2$ the number of columns of $F^2$.  It follows from \cref{rem:uniquedecompositionbygeneratorsimpliesthatthegeneratorsareextreme} and property (b) that $\set{F^1_{\star 1},\ldots,F^1_{\star K_1}}$  and  $\set{F^2_{\star 1},\ldots,F^2_{\star K_2}}$ are minimal sets. It follows from \cref{cor:atmostoneminimalset} that the sets are equal. In particular, $K_2=K_1$ and there is a permutation $\pi:\set{1,\ldots,K_1}\to \set{1,\ldots, K_1}$ so that $F_{\star k}^2 = F^1_{\star \pi(k)}$. It follows from property (b) and \cref{lem:uniqueconvexcombinationintermsofindependentvectors} that each  column of $\Pi$ has a unique decomposition in terms of the columns of $F^1$ (or of $F^2$). So $Q^2_{k\star}=Q^1_{\pi(k)\star}$, for all $k\in\set{1,\ldots,K}$. So $(F^1,Q^1)\sim (F^2,Q^2)$. So $\sM$ is identifiable.  
\end{proof}
  
 In \cref{thm:identifiabilityanchorinQ}, where all elements $\sF_K^{\text{in}}$ are allowed,
 I use the anchor condition on 
 $Q$, to guarantee that $\co(F^1)=\co(F^2)$, whenever 
 $F^1Q^1=F^2Q^2$. One could 
 imagine other conditions on $Q$ 
 or $F$ that lead to $\co(F^1)=\co(F^2)$ whenever 
 $F^1Q^1=F^2Q^2$. But it might 
 as well be possible that that the 
 anchor condition is necessary. 
\Cref{thm:identifiabilityanchorinQcounterexample} at least 
shows that one cannot deviate much from the anchor condition.

Instead of using the theory of convex spaces, one can take the approach of cones. This leads to a similar result: 
\begin{theorem}
    Let $\sM$ be a model so that 
    \begin{enumerate}
        \item $K(F)<M$ for all $(F,Q)\in \sM$,
        \item for every $(F,Q)\in\sM$, $\sF_{K(F)}^{\text{an}}\x \set Q\subseteq \sM$. 
    \end{enumerate}
     Then $\sM$ is identifiable if and only if 
    \begin{enumerate}[(a)]
    \item for all $(F^1,Q^1),(F^2,Q^2)\in \sM$ so that $F^1Q^1=F^2Q^2$, we have that $\cone(Q^1)=\cone(Q^2)$. 
    \item For every $(F,Q)\in\sM$, the rows of $Q$ are linearly independent.     
    \end{enumerate}
\end{theorem}
\begin{proof}
Let $\sM$ be an identifiable model. It follows from  \cref{thm:identfiabilitynecessaryconditions} that (a) holds.
Let $(F,Q)\in \sM$ and let $K$ be the number of columns of $F$. Suppose that the rows of $Q$ are not linearly independent. So there is a nonzero vector $v\in\re^K$ so that $v'Q=0$. For some $\delta>0$, $\max_i |\delta v_i|<1/2$. Define 
\begin{align*}
    F^1 =\begin{pmatrix}
        e'/2\\ I_K \\ e_1'\\\vdots \\ e_1' 
    \end{pmatrix}, \quad F^2 =  \begin{pmatrix}
        e'/2+\delta v'\\ I_K \\ e_1'\\\vdots \\ e_1' 
    \end{pmatrix},
\end{align*}
where the last $M-K-1\ge 0$ rows are $e_1'$ and $I_K$ is the $K\x K$ identity matrix.  
Then by property (1) $F^1,F^2 \in \sF_{K(F)}^{\text{an}}$, and by property (2) $(F^1,Q),(F^2,Q)\in \sM$. We have that $F^1Q=F^2Q$, but $(F^1,Q)\not\sim (F^2,Q)$. Contradiction. So the rows of $Q$ are linearly independent. 

Now let $\sM$ be a model that satisfies (a) and (b). I'll show that $\sM$ is identifiable. Let $(F^1,Q^1),(F^2,Q^2)\in \sM$ be so that $F^1Q^1=F^2Q^2$. Let $K_1$ be the number of columns of $F^1$ and $K_2$ be the number of columns of $F^2$.  It follows from \cref{lem:coneuniquedecompositionifandonlyifelementsarelinearlyindependent} in combination with \cref{lem:aconegeneratedbyAwithuniquedecompositionthenAisminimalandconsistofextremepoints} that $\set{Q_{\star 1}^1,\ldots,Q_{\star K_1}^1}$ and $\set{Q_{\star 1}^2,\ldots,Q_{\star K_2}^2}$ are two minimal sets of $\cone(Q^1)=\cone(Q^2)$. It follows from \cref{lem:uniquenessofminimalsetthatgeneratescone} that $K_2=K_1$ and there are $\delta_1,\ldots,\delta_{K_1}>0$ and a permutation $\pi:\set{1,\ldots,K_1}\to \set{1,\ldots,K_1}$ so that $Q_{k\star }^2 = \delta_k Q_{ \pi(k)\star}^1$, for all $k\in\set{1,\ldots,K_1}$. Define $\delta$ by $\delta'=(\delta_{\pi^{-1}(1)},\ldots,\delta_{\pi^{-1}(K_1)})$, where $\pi^{-1}$ denotes the inverse map of $\pi$. Then $\delta'Q^1=e'Q^2=e'Q^1=e'$. As the rows of $Q^1$ are linearly independent, it follows that $\delta=e$. So $Q_{\star k}^2= Q_{\star \pi(k)}^1$ for all $k\in\set{1,\ldots,K_1}$. According to \cref{lem:coneuniquedecompositionifandonlyifelementsarelinearlyindependent}, as the rows of $Q^1$ are linearly independent, each element of $\cone(Q^1)$ has a unique decomposition in terms of the rows of $Q^1$. It follows that $F^2_{\star k}= F^1_{\star \pi(k)}$. In particular $(F^1,Q^1)\sim (F^2,Q^2)$. So $\sM$ is identifiable. 
\end{proof}

 In \cref{thm:identifiabilityanchorinF} uses the anchor condition on $F$ to guarantee that $\cone(Q^1)=\cone(Q^2)$ whenever $F^1Q^1=F^2Q^2$, for $(F^1,Q^1),(F^2,Q^2)$ in the model. Also, here, it is not clear to me whether the anchor condition is essential or if, under some weaker conditions, we still have $\cone(Q^1)=\cone(Q^2)$ whenever $F^1Q^1=F^2Q^2$. But also here we see in \cref{thm:identifiabilityanchorinFcounterexample} that we cannot deviate much from the anchor condition. 

While the independence requirements seem not a significant restriction, the anchor requirements seem pretty restrictive. It would be nice to either find a model that doesn't require it or prove that it is necessary, one way or another.

\section{Unadmixted}

This section considers the non-admixted case, so each individual inherits their genome only from one ancestor. In mathematical terms, $Q_{ik}\in \set{0,1}$. I find sufficient conditions that are also necessary for identifiability in this particular case. 

\begin{theorem}
    Let $\sF_K^{\text{d}}\subseteq \sF_K$ be all $F\in\sF_K$ so that the columns of $F$ are mutually different. Let $\sQ^{\text{ua}}_K\subseteq \sQ_K$ be the set of all matrices $Q\in\sQ$ so that each $i\in\set{1,\ldots,N}$ there is a $k\in\set{1,\ldots,K}$ so that $Q_{ik}=1$ and $Q_{i\ell}=0$ for all $\ell\neq k$. Moreover, for each $k$, there is a column $i$ in $Q$ so that $Q_{\star i}=e_k$. Then $\sM'''=\bigcup_{K=1}^\infty \sF_K^{\text d}\x \sQ_K^{\text{ua}}=\bigcup_{K=1}^N \sF_K^{\text d}\x \sQ_K^{\text{ua}}$ is identifiable.
\end{theorem}
\begin{proof}
As $\sQ_K^{\text{ua}}=\leeg$ when $K>N$, it follows that $\bigcup_{K=1}^\infty \sF_K^{\text d}\x \sQ_K^{\text{ua}}=\bigcup_{K=1}^N \sF_K^{\text d}\x \sQ_K^{\text{ua}}$.
Let $(F^1,Q^1),(F^2,Q^2)\in\sM'''$ be such that $F^1Q^1=F^2Q^2$. Let $K_1$ be the number of columns of $F_1$ and $K_2$ the number of columns of $F_2$. Note that by the property that each column of $Q^1$ is either $e_1,\ldots,e_{K_1-1}$ or $e_{K_1}$, and that for each $k\in\set{1,\ldots,K_1}$, $e_{k}$ is a column in $Q^1$ (and similar for $Q^2$), that the columns of $\Pi$ are columns of $F^1$ and of $F^2$. It follows that \[
    \set{F_{\star k}^1:k\in\set{1,\ldots,K_1}}= \set{F_{\star k}^2:k\in\set{1,\ldots,K_2}}= \set{\Pi_{\star i}^1:i\in\set{1,\ldots,N}}.  
    \]
    As both $F^1$ and $F^2$ have no identical columns, it follows that $K_2=K_1$ and the set above has $K_1$ elements, and there is a bijection $\pi:\set{1,\ldots,K_1}\to\set{1,\ldots,K_1}$ so that $F_{\star k}^2 = F_{\star \pi(k)}^1$. Let $i\in\set{1,\ldots,N}$. Then there is a unique $k$ so that $\Pi_{\star i}=F_{\star k}^2=F_{\star \pi(k)}^1$. As all columns in $F^1$ are different (and so for $F^2$), we have that $Q_{\star i}^1 = e_k$ and $Q_{\star i}^2=e_{\pi(k)}$. So $(F^1,Q^1)\sim (F^2,Q^2)$.  
\end{proof}
The following theorem shows that under the non-admixability assumption, the sets $\sF^{\text d}_K$ and $\sQ^{\text{ua}}_K$ cannot be enlarged so that their product space is still an identifiable model.
\begin{theorem}
  Let $K\ge 2$ and $N>K$.  Let $F\in\sF_K$ be such so that there are (at least) two identical columns.  Then there are $Q^1,Q^2\in\sQ_K^{\text{ua}}$ so that $FQ^1=FQ^2$, but $(F,Q^1)\not\sim (F,Q^2)$. 
  
  On the other hand, let $K\ge 2$ and $N\ge 1$, and let $Q\in\sQ_K\weg \sQ_K^{\text{ua}}$ be so that every column of $Q$ is of the form $e_i$, but for some $k$, $e_k$ is not a column of $Q$. Then for every $F\in\sF^{\text{ua}}_K$ there is a $F^2\in\sF^{\text{ua}}_K$ so that $FQ=F^2Q$, but $(F,Q)\not\sim (F^2,Q)$.
\end{theorem}
\begin{proof}
Let $k,\ell\in\set{1,\ldots,K},k\neq \ell$ so that the $k$th and $\ell$th column of $F$ are identical. Consider \begin{align*}
    Q^1 = & (I, e_k,\ldots,e_k),\\
    Q^2 = & (I,e_\ell,\ldots,e_\ell),
\end{align*} 
then $FQ^1=FQ^2$, but $(F,Q^1)\not\sim (F,Q^2)$. 
    
    For the second statement, let $F^2\in\sF_K^{\text d}$ be any matrix so that $F^2_{\star \ell}=F_{\star \ell}$ for all $\ell\neq k$, and  $F^2_{\star k}$ is a column different to all columns in $F$. Then $FQ=F^2Q$, as $e_k$ is not a column in $Q$, but $(F,Q)\not\sim (F^2,Q)$.  
\end{proof}
So the unadmixed case is easy: provided $N>K$, our sufficient conditions are also necessary.  So we see here again the phenomenon that shrinking the sets from $\sQ_K^{\text{in}}$ to $\sQ_K^{\text{an}}$ to $\sQ_K^{\text{ua}}$ allows us to enlarge the sets of allowable $F\in\sF_K$, from $\sF_K^{\text{an}}$ to  $\sF_K^{\text{in}}$ to $\sF_K^{\text{d}}$.  

\appendix

\section{Convex sets}\label{sec:convexsets}

This section develops the theory on convex sets required for the proofs. 
A convex set is a subset of a real linear space so that for every $x,y\in C$ and $\lambda\in[0,1]$, $\lambda x + (1-\lambda)y\in C$.  
Let $S$ be a non-empty set of a real vector space. Then we define the convex set \textbf{generated} by $S$ as 
\begin{equation}\label{def:convexhull}
\co(S)= \set{\sum_{k=1}^N\lambda_k x_k: N\in\NN, x_1,\ldots,x_N\in S, \lambda_1,\ldots,\lambda_N\ge 0, \sum_{k=1}^N\lambda_k=1}. 
\end{equation}

\subsection{Unique decompositions}

In this subsection I am interested under what conditions the elements of $\co(S)$ have a unique convex decomposition in terms of the elements of $S$. So when $\lambda_1,\ldots,\lambda_K,\mu_1,\ldots,\mu_K\ge 0$ and $\sum_{k=1}^K\lambda_k=\sum_{k=1}^K \mu_k=1$, and $\sum_{k=1}\lambda_kx_k=\sum_{k=1}^K \mu_kx_k$, $x_1,\ldots,x_k\in S$, is then also $\lambda_i=\mu_i$ for all $i$? I'll develop precise necessary and sufficient conditions for this. 

\begin{definition}\label{def:uniqueconvexcombination}
	Let $C$ be the convex hull of $v_1,\ldots,v_m$. An element $v\in C$ has a unique convex combination of elements $v_1,\ldots,v_m$ when $\lambda_1,\ldots,\lambda_m,\mu_1,\ldots,\mu_m\ge 0,\sum_{i=1}^m\mu_i=\sum_{i=1}^m\lambda_i=1$ and $v=\lambda_1v_1+\ldots+\lambda_mv_m=\mu_1v_1+\ldots+\mu_mv_m$ impies $\lambda_i=\mu_i$, for all $i=1,\ldots,m$. 
\end{definition}

\begin{lemma}\label{lem:uniqueconvexcombinationintermsofindependentvectors}
	Let $V$ be a real vector space. Let $k\in\NN$. Let $v_1,\ldots,v_{k+1}\in V$ and let $C$ be the convex hull of $v_1,\ldots,v_{k+1}$. 
	Then each element of $C$ has a unique convex combination of elements of $v_1,\ldots,v_{k+1}$ if and only if  $v_1-v_{k+1},\ldots,v_{k}-v_{k+1}$ are linearly independent.
\end{lemma}
\begin{proof}
First we prove that when $v_1-v_{k+1},\ldots,v_k-v_{k+1}$ are linearly independent, that each element of $C$ has a unique convex combination of elements $v_1,\ldots,v_{k+1}$.

Let $v\in C $ and let $v=\lambda_1v_1+\ldots+\lambda_{k+1}v_{k+1}=\mu_1v_1+\ldots+\mu_{k+1}v_{k+1}$ be convex combinations of $v$. Then, making use of the fact that $\lambda_1+\ldots+\lambda_{k+1}=\mu_1+\ldots+\mu_{k+1}=1$, $
v-v_{k+1}=\lambda_1(v_1-v_{k+1})+\ldots+\lambda_{k+1}(v_{k+1}-v_{k+1})=\mu_1(v_1-v_{k+1})+\ldots+\mu_{k+1}(v_{k+1}-v_{k+1})$,
so 
\[\lambda_1(v_1-v_{k+1})+\ldots+\lambda_{k}(v_{k}-v_{k+1})=\mu_1(v_1-v_{k+1})+\ldots+\mu_{k}(v_{k}-v_{k+1}).\]
It follows from the fact that $v_1-v_{k+1},\ldots,v_k-v_{k+1}$ are linearly independent, that $\lambda_i=\mu_i$ for all $i=1,\ldots,k$. Finally, $\lambda_{k+1}=1-\lambda_1-\ldots-\lambda_k=1-\mu_1-\ldots-\mu_k=\mu_{k+1}$. So  $v$ has a unique convex combination.

For the proof in the other direction, suppose $v_1-v_{k+1},\ldots,v_k-v_{k+1}$ are not linearly independent. We will show, that there is an element in the convex hull of $v_1,\ldots,v_{k+1}$ that does not have a unique convex combination.

From the linear dependence of  $v_1-v_{k+1},\ldots,v_k-v_{k+1}$ follows that there are $\alpha_1,\ldots,\alpha_k$, not all zero, so that $
\alpha_1(v_1-v_{k+1})+\ldots+\alpha_k(v_k-v_{k+1})=0$. 
Let $I=\set{i:\alpha_i>0}$ and $J=\set{i:\alpha_i\le 0}$.  So 
\[
\sum_{i\in I}\alpha_i(v_i-v_{k+1})=\sum_{i\in J}-\alpha_i(v_i-v_{k+1}).
\]
As at least one $\alpha_i\neq 0$, $i\in\set{1,\ldots,k}$, at least one of $\sum_{i\in I}\alpha_i$ or $\sum_{i\in J}-\alpha_i$ is positive, and both are non-negative. Let $M=\max\set{\sum_{i\in I}\alpha_i, \sum_{i\in J}-\alpha_i}>0$. Let $\beta=M-\sum_{i\in I}\alpha_i$ and $\gamma=M-\sum_{i\in J}-\alpha_i$. Note that $\beta,\gamma\ge 0$, and that 
$\beta+\sum_{i\in I}\alpha_i=\gamma+\sum_{i\in J}-\alpha_i=M$. As $v_{k+1}-v_{k+1}=0$, we have 
\[
\frac\beta M(v_{k+1}-v_{k+1})+\sum_{i\in I}\frac{\alpha_i}M(v_i-v_{k+1})=\frac\gamma M(v_{k+1}-v_{k+1})+\sum_{i\in J}\frac{-\alpha_i}M(v_i-v_{k+1}).
\]
Using that 
$\frac\beta M+\sum_{i\in I}\frac{\alpha_i}M=\frac\gamma M+\sum_{i\in J}\frac{-\alpha_i}M=1$, adding $v_{k+1}$ on both sides gives
\[
\frac\beta Mv_{k+1}+\sum_{i\in I}\frac{\alpha_i}Mv_i=\frac\gamma Mv_{k+1}+\sum_{i\in J}\frac{-\alpha_i}Mv_i.
\]
As $I$ and $J$ are disjoint, and at least one of $\alpha_i\neq 0$, it follows that this are two different convex combinations of $v_1,\ldots,v_{k+1}$ of the same element $\frac\beta Mv_{k+1}+\sum_{i\in I}\frac{\alpha_i}Mv_i$. 
\end{proof}

\begin{lemma}\label{lem:twoconvexcombinationsimpliesinfiniteconvexcombinations}
	Let $V$ be a vector space, and $C$ the convex hull of $v_1,\ldots,v_m\in V$. When $v\in C$ has two different convex combinations of $v_1,\ldots,v_m$, then $v$ has infinitely many convex combinations of $v_1,\ldots,v_m$. 
\end{lemma}
\begin{proof}
	Suppose 
	 \[
		v=\sum_{i=1}^m \lambda_iv_i=\sum_{i=1}^m \mu_iv_i
		\]
	are two different convex combinations of $v$. So for some $i_0\in\set{1,\ldots,m}$, $\lambda_{i_0}\neq \mu_{i_0}$. Let $\alpha\in[0,1]$. Note that 
	 \[
		v=\sum_{i=1}^m (\alpha \lambda_i+(1-\alpha)\mu_i)v_i=:\sum_{i=1}^m \nu_i(\alpha)v_i, 
		\]
		is also a convex combination of $v$. When $\alpha_1\neq\alpha_2$, $\nu_{i_0}(\alpha_1)-\nu_{i_0}(\alpha_2)=(\alpha_1-\alpha_2)(\lambda_{i_0}-\mu_{i_0})\neq0$. Hence there are infinitely many convex combinations of $v$.   
\end{proof}

\begin{definition}\label{def:openconvexcombination}
	Let $C$ be a convex set. A convex combination \[
	v=\sum_{i=1}^m \lambda_iv_i
	\]
	is open when for all $i\in\set{1,\ldots,m}$, $\lambda_i>0$. 
\end{definition}

\begin{definition}
	Let $V$ be a real vector space and let $v_1,\ldots,v_m\in V$. We define the open convex set generated by $v_1,\ldots,v_m$ to be the set of all open convex combinations of $v_1,\ldots,v_m$. 
\end{definition}

Note that when $V=\re^n$ and $m>2$, then the open convex set is also open in the topological sense of the word. This is not the case when $m=1$. 

\begin{remark}
    To distinguish between the open convex hull and the ``usual" convex hull, we call the later sometimes the ``closed convex hull".  
\end{remark}

\begin{lemma}\label{lem:propertiesopenconvexset}
	Let $V$ be a real vector space and let $C^\circ$ be the open convex set generated by $v_1,\ldots,v_m\in V$. Let $C$ be the convex set generated by $v_1,\ldots,v_m$. Then $C^\circ$ is convex and $\leeg\neq C^\circ\subseteq C$. 
\end{lemma}
\begin{proof}
	It is obvious that $C^\circ$ is contained in the convex set generated by $v_1,\ldots,v_m$. We have that $(1/m)v_1+\ldots+(1/m)v_m\in C^\circ$, so $C^\circ$ is not empty. 
	
	Let $v=\sum_{i=1}^m \lambda_iv_i, w=\sum_{i=1}^m \mu_iv_i\in C^\circ$, $\lambda_i,\mu_i>0$ for all $i$. Let $\alpha\in [0,1]$. Then 
	\[
	\alpha v+(1-\alpha)w=\sum_{i=1}^m (\alpha\lambda_i+(1-\alpha)\mu_i)v_i. 
	\]
	Note that $\sum_{i=1}^m (\alpha\lambda_i+(1-\alpha)\mu_i)=1$, and $\alpha\lambda_i+(1-\alpha)\mu_i>0$, for all $i\in\set{1,\ldots,m}$. Hence $\alpha v+(1-\alpha)w\in C^\circ$. So $C^\circ$ is convex. 
\end{proof}

\begin{lemma}
Let $V$ be a vector space and let $C^\circ$ (resp. $C$) be the open (resp. closed) convex set generated by $v_1,\ldots,v_m\in V$. 
 The following statements are equivalent:
	\begin{enumerate}[(i)]
		\item There is an element $v\in C$ that does not have a unique convex combination of $v_1,\ldots,v_m$. 
		\item Every element of $v\in C^\circ$ does not have a unique convex combination of $v_1,\ldots,v_m$. 
		\item For every element $v\in C^\circ$ there are infinitely many convex combinations of $v_1,\ldots,v_m$. 
	\end{enumerate}
\end{lemma}
\begin{proof}
Obviously, (iii)$\implies$(ii). As $C^\circ$ is not empty and contained in $C$ (\cref{lem:propertiesopenconvexset}), (ii)$\implies$(i). The implication (ii)$\implies$(iii) follows from \cref{lem:twoconvexcombinationsimpliesinfiniteconvexcombinations}. We are only left to prove (i)$\implies$(ii). Let $v\in C$ be an element so that \[
		v=\sum_{i=1}^m \lambda_iv_i=\sum_{i=1}^m \mu_iv_i
		\]
		are two different convex combinations of $v$. 
		Note that 
		\[
		0 =\sum_{i=1}^m (\lambda_i-\mu_i)v_i. 
		\]
		Let $w\in C^\circ$ have an open convex combination \[
		w=\sum_{i=1}^m \nu_iv_i. 
		\]
		Let $\alpha=\min_i\nu_i>0$. As $\lambda_i-\mu_i\ge -1$, $\nu_i+\alpha(\lambda_i-\mu_i)\ge 0$, for all $i$, and $\sum_{i=1}^m(\nu_i+\alpha(\lambda_i-\mu_i))=\sum_{i=1}^m\nu_i+\alpha\sum_{i=1}^m(\lambda_i-\mu_i)=1+0=1$. 
		So 
		\[
		w = \sum_{i=1}^m (\nu_i+\alpha(\lambda_i-\mu_i))v_i. 
		\]
	is another convex combination of $w$, because for at least one $i\in\set{1,\ldots,m},\lambda_i\neq \mu_i$ and $\alpha>0$. 
\end{proof}

\subsection{Convex hulls}

In this subsection, I am interested in the smallest subsets $S$ of a convex set $C$, so that $S$ generates $C$ (in formula's $C=\co(S)$). It turns out that not every convex set $C$ has a smallest set $S$ that generates $C$, and even in cases where it happens, there is not always a unique decomposition in terms of the elements of $S$.  

\begin{definition}
	Let $S$ be a subset of a vector space, and let $C$ be the convex space generated by $S$. We call $S$ minimal, when for every $x\in S$, $C\neq \co(S\weg\set x)$. 
\end{definition}

Not every convex set has a minimal generating set. 
\begin{example}
    Consider the real numbers $\re$, which is a convex set, and let $S\subseteq \re  $ be a set that generates  $\re$.
    First note that $S$ is infinite, as otherwise $r=\max_{x\in S}|x|<\infty$ and $\co(S)\subseteq[-r,r]\neq \re$. It follows that there are $x,y,z\in S$ so that  $x<y<z$. Note that $y$ is a convex combination of $x$ and $z$, so $S\weg\set y$ also generates $\re$. So $S$ is not minimal. We chose $S$ arbitrary, so $\re$ has no minimal generating set. 
\end{example}

However, if $C$ is generated by a finite set $S$, then there exists a minimal set. 
\begin{lemma}\label{lem:finitegeneratedconvexsethasminimalset}
    Let $S$ be a finite non-empty subset of a real vector space. Let $C=\co(S)$. Then $C$ has a minimal set.
\end{lemma}
\begin{proof}
    Define $S_0=S$ and until $S_i$ is minimal, set $S_{i+1}=S_{i}\weg\set x$, where $x\in S_i$ is an element so that $C=\co(S_i)=\co(S_i\weg\set x)$, which exist when $S_i$ is not minimal. As $S_0$ is finite, and $S_{i+1}$ has one element less than $S_i$, this algorithm is destined to terminate after $i_0\ge 0$ steps. Note that $S_{i_0}$ is not empty and $C=\co(S_{i_0})$ and $S_{i_0}$ is minimal. 
\end{proof}

\begin{definition}\label{def:extremeset}
	Let $C$ be convex and $x\in C$. We call $x$ extreme, when there are no $y,z\in C$, $y\neq z$ and $\alpha\in(0,1)$ so that $x=\alpha y+ (1-\alpha)z$. 
\end{definition}

\begin{example}
    In $C=\set{(x,y)\in\re^2:0\le x\le 1, 0\le y \le 1}$, $(0,0),(0,1), (1,0)$ and $(1,1)$ are extreme points. 
\end{example}

\begin{lemma}\label{lem:theminimalsetisthesetofextrema}
		Let $C$ be a convex set generated by a minimal set $S$.  Then $S$ is the set of all extrema of $C$. 
	\end{lemma}
	\begin{proof}
For an extremum $x\in C$, there are no $y,z\in C,y\neq z$ and $\alpha\in(0,1)$ so that $x=\alpha y +(1-\alpha)z$. So $\co(S\weg \set x)$ does not contain $x$. Hence $x\in S$. 
	
Suppose $x\in S$ is not extreme. Then there are $y,z\in C,x\neq y$ and $\alpha\in(0,1)$ so that $x=\alpha y+(1-\alpha)z$. Then there are mutually different elements $x_1,\ldots,x_m\in S$, $m\ge 1$, so that $y,z$ are convex combinations 
\[
y=\sum_{i=1}^m \beta_ix_i,\en z=\sum_{i=1}^m \gamma_ix_i. 
\]	
Note that we can choose this $x_i$ so that at least one of $\beta_i$ or $\gamma_i$ is positive, for every $i\in\set{1,\ldots,m}$. 
So \[
x= \sum_{i=1}^m (\alpha\beta_i+(1-\alpha)\gamma_i)x_i. 
\]
If all $x_i\neq x$, then $x$ is a convex combination of other elements of $S$, and so $C=\co(S\weg\set x)$, so $S$ is not minimal. Contradiction. So $x$ is equal to some $x_i$. 
After relabelling, if necessary, we may assume $x=x_1$. As $y\neq z$, either $\beta_1<1$ or $\gamma_1<1$, or both. In particular $m\ge 2$. We already assumed that $\beta_1>0$ or $\gamma_1>0$. So $0<\alpha\beta_1+(1-\alpha)\gamma_1<1$. So 
\[
(1-(\alpha\beta_1+(1-\alpha)\gamma_1))x =  \sum_{i=2}^m (\alpha\beta_i+(1-\alpha)\gamma_i)x_i.
\]
Note that $\sum_{i=2}^m (\alpha\beta_i+(1-\alpha)\gamma_i)=1- (\alpha\beta_1+(1-\alpha)\gamma_1)$, so 
\[
x =  \sum_{i=2}^m \frac{\alpha\beta_i+(1-\alpha)\gamma_i}{1-(\alpha\beta_1+(1-\alpha)\gamma_1)}x_i,
\]
is a convex combination of elements from $S\weg \set x$. So $C=\co(S\weg\set x)$, so $S$ is not minimal. Contradiction. As this were all posibilities, it follows that all elements of $S$ are extreme. 
\end{proof}

However, a set of extreme points does not necessarily generate the convex set. 

\begin{example} Consider \[
	C = \set{(x,y)\in \re ^ 2: x^2+y^2\le 1, \text{ when }x,y\ge 0, \text{ otherwise }x^2+y^2<1}. 
	\]
	Then $C$ is convex, and the set of extreme points is 
	\[
	E = \set{(x,y):x^2+y^2=1, x\ge 0, y\ge 0}. 
	\]
	But $C$ is not generated by $E$. 
\end{example}

A corollary to \cref{lem:theminimalsetisthesetofextrema} is 
\begin{corollary}\label{cor:atmostoneminimalset}
	A convex set $C$ has at most one minimal set. 
\end{corollary}
\begin{proof}
    If $C$ has a minimal set $S$, then $S$ is the set of extrema. So $S$ is uniquely determined.
\end{proof}
So when a minimal set exists, it is unique, which allows us to speak about the minimal set. 

\begin{lemma}\label{lem:extremegeneratingsetisminimal}
    When $C$ is a convex set generated by its set $E$ of extrema, then $E$ is the minimal set. 
\end{lemma}
\begin{proof}
    Suppose $E$ is not minimal, then there is an $x\in E$ so that $C$ is generated by $E\weg\set x$. So there are $x_1,\ldots,x_m\in E\weg \set x$ and $\lambda_1,\ldots,\lambda_m>0$, $\sum_{i=1}^m \lambda_i=1$ so that $x=\sum_{i=1}^m \lambda_ix_i$. But then $x$ is not extreme. Contradiction. 
\end{proof}

A corollary to \cref{lem:theminimalsetisthesetofextrema,lem:extremegeneratingsetisminimal} is 
\begin{corollary}\label{cor:characterisationofminimalsetsintermsofextremepointsandgenerating}
Let $C$ be a convex set generated by $S\subseteq C$. Then $S$ is minimal if and only if $S$ is the set of all extreme points.  
\end{corollary}

As a corollary to \cref{lem:finitegeneratedconvexsethasminimalset,cor:atmostoneminimalset} we have that every finitely generated convex set has a unique minimal set. 
\begin{corollary}
    Let $S$ be a finite non-empty subset of a real linear space. Then $\co(S)$ has a unique minimal set. 
\end{corollary}

However, not every element in a convex set $C$ generated by a minimum set $S$ has a unique decomposition in elements in $S$. 
\begin{example}
 Take for instance $C=\set{(x,y)\in\re^2:x^2+y^2\le 1}$, which has minimal set $S=\set{(x,y)\in\re^2:x^2+y^2=1}$. Then \[
	(0,0)=\frac12(-1,0)+\frac12(1,0)\en (0,0)=\frac12(0,-1)+\frac12(0,1). 
	\]
\end{example}
However, if $C$ is generated by $S=\set{x_1,\ldots,x_m}$ and every element in $C$ has a unique decomposition in terms of elements of $S$, then $S$ is minimal:
\begin{lemma}\label{lem:uniquedecompositionareextreme}
	Let $C$ be a convex set generated by $S=\set{x_1,\ldots,x_m}$. If every element in $C$ has a unique decomposition in terms of $S$, then $S$ is a minimal set. 
\end{lemma}
\begin{proof}
	Let $x_j\in S$. Suppose there are $y,z\in C$ and $\alpha\in (0,1)$ so that \[
	x_j= \alpha y + (1-\alpha)z. 
	\] 
	Then $y$ and $z$ have convex decompositions \[
	y = \sum_{i=1}^m \beta_ix_i,\en z = \sum_{i=1}^m \gamma_ix_i.
	\]
	So \[
	x_j= \sum_{i=1}^m (\alpha \beta_i+(1-\alpha)\gamma_i)x_i,
	\]
	is a convex decomposition of $x_j$ in terms of $x_1,\ldots,x_m$. 
As the convex decompositions are unique, $\beta_i=\gamma_i=0$ for all $i\neq j$, and $\beta_j=\gamma_j=1$, so $y=z$, so $x_j$ is extreme. So by \cref{cor:characterisationofminimalsetsintermsofextremepointsandgenerating} $S$ is a minimal set.  
\end{proof}

A corollary to \cref{lem:uniqueconvexcombinationintermsofindependentvectors,lem:uniquedecompositionareextreme} is 

\begin{corollary}\label{rem:uniquedecompositionbygeneratorsimpliesthatthegeneratorsareextreme}
Let $v_1,\ldots,v_k$ be vectors in a vectors space and let $C$ be the convex space generated by $v_1,\ldots,v_k$. When $v_1-v_k,\ldots,v_{k-1}-v_k$ are linearly independent, then $v_1,\ldots,v_k$ are extreme points of $C$ and $\set{v_1,\ldots,v_k}$ is a minimal set. 
\end{corollary}

\section{Cones}\label{sec:cones}

In this section, I develop a theory similar for cones as I did for the convex spaces. A cone is subset $K$ of a real vector space so that for all $x,y\in K$, and $\alpha\ge 0$, also $\alpha x$ and $x+y$ are in $K$ and if $-K=\set{-x:x\in K}$, then $K\cap (-K)=\set 0$. The cone generated by a set $S$ is the set 
\[
\cone(S)=\set{\sum_{k=1}^N \alpha_kx_k:N\in\NN, \alpha_1,\ldots,\alpha_N\ge 0},
\]
provided $\cone(S)\cap(-\cone(S))=\set 0$. I am again interested when every element of $\cone(S)$ can be written as a unique decomposition 
\[
\sum_{k=1}^N \alpha_kx_k\]
of elements $x_1,\ldots,x_N\in S$. I am also interested in smallest subsets $S\subseteq K$ that generate $K$. 

The following definition is taken from \cite[Definition 1.1.1]{KalauchvanGaans2018}. 

\begin{definition}\label{def:coneandwedge}
    A wedge $K$ is a subset of a vector space so that when $x,y\in K$ then also $x+y\in K$ and if $\alpha\ge 0$ then also $\alpha x \in K$. If, additionally, $-K= \set{-x : x\in K}$, and $K$ satisfies $K\cap (-K)=\set 0$, then we call $K$ a cone.  
\end{definition}

\begin{definition}\label{def:wedgeorconegeneratedbysetA}
    Let $A$ be a subset of a vector space. We define the wedge generated by $A$ as the set \begin{align*}
     \we(A)=   \set{\sum_{i=1}^n \alpha_i x_i: n\in\NN, \alpha_1,\ldots,\alpha_n\ge 0, x_1,\ldots,x_n\in A}.
    \end{align*}
    Note that $A$ is a wedge. If additionally  $\we(A)$ is a cone, then we say that $\cone(A):=\we(A)$ is the cone generated by $A$. 
\end{definition}
Note that the wedge generated by the empty set is $\set 0$. 

\begin{definition}\label{def:uniquedecompositionwedgeorcone}
    Let $A$ be a subset of a real vector space. Let $x\in \we(A)$. Then $x$ has a unique decomposition in terms of $A$, if for mutually different elements $x_1,\ldots,x_m\in A$, and arbitrary  $\alpha_1,\ldots,\alpha_n,\beta_1,\ldots,\beta_n\ge 0$, 
    \begin{align*}
        &x=\sum_{i=1}^n \alpha_i x_i = \sum_{i=1}^n \beta_i x_i,     \end{align*}
    implies $\alpha_i=\beta_i$, for all $i\in\set{1,\ldots,n}$. 
\end{definition}

\begin{lemma}
    Let $A$ be a subset of a real linear space. Then the following are equivalent:
    \begin{enumerate}
        \item $0\in \we(A)$ does not have a unique decomposition,
        \item every element in $\we(A)$ does not have a unique decomposition.  
    \end{enumerate}
\end{lemma}
\begin{proof}
    Obviously, 2 implies 1. Now assume 1. Then there is an $N\in \NN$ and there are  positive scalars $\lambda_1,\ldots,\lambda_N$ and $x_1,\ldots,x_N\in A$ so that \[
    0 = \sum_{k=1}^N \lambda_k x_k. 
    \]
    Let $x\in \we(A)$. Then there are $M\in \NN$ and $\mu_1,\ldots,\mu_M$, $y_1,\ldots,y_M\in A$ so that \[
    x = \sum_{k=1}^M \mu_ky_k.
    \]
    Then \[
    x=x+0=\sum_{k=1}^M \mu_ky_k + \sum_{k=1}^N \lambda_k x_k
    \]
    is another representation of $x$. 
\end{proof}

\begin{lemma}
    If $0\in\we(A)$ has a uniquely decomposition in terms of $A$, then $\we(A)$ is a cone. 
\end{lemma}
\begin{proof}
    Let $x\in \we(A)\cap(-\we(A))$. So $x,-x \in \we(A)$. Let 
    \begin{align*}
        x = \sum_{i=1}^n \alpha_ix_i,\\
        -x = \sum_{i=1}^n \beta_ix_i,\\
        \alpha_1,\ldots,\alpha_n,\beta_1,\ldots,\beta_n\ge 0, \quad x_1,\ldots,x_n\in A
    \end{align*}
    be decompositions of $x$ and $-x$. 
    It follows that 
    \begin{align*}
        0 = x+-x = \sum_{i=1}^n (\alpha_i+\beta_i)x_i 
    \end{align*}
    is the unique decomposition of 0. So $\alpha_i+\beta_i=0$ for all $i\in\set{1,\ldots,n}$. So $\alpha_i=\beta_i=0$ for all $i\in\set{1,\ldots,n}$. So $x=0$. So $\we(A)$ is a cone. 
\end{proof}

\begin{lemma}\label{lem:coneuniquedecompositionifandonlyifelementsarelinearlyindependent}
 Let $A$ be a subset of a real vector space. Then every element in $\we(A)$ has a unique decomposition in terms of $A$ if and only if all elements in $A$ are linearly independent. 
\end{lemma}
\begin{proof}
    Suppose the elements in $A$ are linearly independent, then every element in $\we(A)$ has a unique decomposition. 
    Suppose $A$ is not linearly independent. So there are  elements $\lambda_1,\ldots,\lambda_n\in\re\weg\set 0$, $n\ge 1$, and $x_1,\ldots,x_n\in A$ so that 
    \begin{align*}
        0 = \sum_{i=1}^n \lambda_i x_i. 
    \end{align*}
    Define $I=\set{i:\lambda_i>0}$ and $J=\set{i:\lambda_i<0}$. At least one of $I$ or $J$ is not empty, so 
    \begin{align*}
        x := \sum_{i\in I} \lambda_i x_i = \sum_{i\in J}(-\lambda_i)x_i 
    \end{align*}
    in $\we(A)$ has no unique decomposition. 
\end{proof}

So when every element in $\we(A)$ has a unique decomposition in terms of $A$, then $\we(A)=\cone(A)$. 

\begin{definition}
    Let $K$ be a wedge and let $x\in K\weg\set 0$. We call $x$ an extreme point when for $y,z\in K\weg\set 0$, $x=y+z$ implies $y=\alpha x, z=\beta x$ for some $\alpha,\beta>0$.  
\end{definition}
\begin{remark}
    When $x$ is an extreme point, then for every $\alpha>0$, $\alpha x $ is also an extreme point. 
\end{remark}

\begin{definition}
    Let $A$ be a subset of a real linear space. We call $A$ minimal, when for every $x\in A$, $\we(A\weg \set x)\neq \we(A)$. 
\end{definition}

Every wedge that is generated by a finite set has a minimum set. 
\begin{lemma}
  Let $K$ be a wedge generated by a finite set $A$. Then $K$ has a minimal set.
\end{lemma}
\begin{proof}
    Let $A_0=A$. Inductively, for $i\ge 0$, if $A_i$ is not minimal, then $A_i$ is not empty, and there is an $x\in A_i$ so that $\we(A_i\weg \set x)=\we(A_i)$. In this case define $A_{i+1}$. If $A_i$ is minimal, set $i_0=i$ and terminate the procedure.  As $A_{i+1}$ has one element less than $A_i$ this algorithm eventually terminates after $i_0\ge 0$ finite steps and $\we(A_{i_0})=K$ and $A_{i_0}$ is minimal. 
\end{proof}

\begin{lemma}\label{lem:aminimalsetthatgeneratesaconeconsistofonlyextremeelements}
    Let $A$ be a minimal set that generates a \textit{cone}, then every element of $A$ is an extremum. 
    
    If $A$ is an nonempty subset of a real linear space, then for every extreme element $x\in \we(A)$, there is some $\alpha>0$, so that $\alpha x\in A$. 
\end{lemma}
\begin{proof}
Let $A$ be a minimal set. Note that for all $x,y\in A$, $x\neq y$:   $x$ is nonzero and, for every $\alpha>0$, $y\neq \alpha x$, as otherwise $A$ is not minimal. 

    Let $x\in A$ and suppose that for some $y,z\in\cone(A)$,  $x = y +z$. Then there are $\alpha_1,\ldots,\alpha_n,\beta_1,\ldots,\beta_n \ge 0$, with $\alpha_i+\beta_i>0$, and $x_1,\ldots,x_n\in A$, $n\ge 1$, so that 
    \begin{align*}
        y = \sum_{i=1}^n \alpha_i x_i, \\
        z = \sum_{i=1}^n \beta_i x_i,
        \intertext{so}
        x = \sum_{i=1}^n (\alpha_i+\beta_i)x_i. 
    \end{align*}
If $n=1$, then $\alpha_1+\beta_1=1$, and $x$ is extreme. Suppose $n\ge 2$. 
If for all $i,$ and for all $\alpha>0$, $x_i\neq \alpha x$, then $x\in \we(A\weg\set x)$, so $A$ is not minimal. Contradiction. 
So, for some $i$, $x_i=\gamma x$ for certain $\gamma>0$. After relabelling, if necessary, we may assume $i=1$. So 
\begin{align*}
    (1- (\alpha_1+\beta_1)\gamma ) x = \sum_{i=2}^n (\alpha_i+\beta_i)x_i.  
\end{align*}
If $ (1- (\alpha_1+\beta_1)\gamma )>0$, then 
\begin{align*}
         x = \sum_{i=2}^n \frac{\alpha_i+\beta_i}{1- (\alpha_1+\beta_1)\gamma }x_i.  
\end{align*}
and $\we(A)=\we(A\weg\set x)$, so $A$ is not minimal. Contradiction. If $ (1- (\alpha_1+\beta_1)\gamma )<0$, then 
\begin{align*}
         -x = \sum_{i=2}^n \frac{\alpha_i+\beta_i}{ (\alpha_1+\beta_1)\gamma -1 }x_i.  
\end{align*}
so $\we(A)$ is not a cone.  Contradiction. 
If $(1- (\alpha_1+\beta_1)\gamma ) = 0$, then 
\begin{align*}
    0 =  \sum_{i=2}^n (\alpha_i+\beta_i)x_i  
\end{align*}
As $x_2\neq 0$ and $\alpha_2+\beta_2>0$, it follows that $n\ge 3$. So 
\begin{align*}
    -x_2 = \sum_{i=3}^n \frac{\alpha_i+\beta_i}{\alpha_2+\beta_2}x_i.
\end{align*}  
So $\we(A)$ is not a cone. Contradiction. As these were all possibilities, we conclude that $x$ is extreme.

Let $x\in \we(A)$ be extreme. Then for all $y,z\in \we(A)$,  $x = y + z$ implies that $y=\alpha x , z=\beta x $ for some $\alpha,\beta\ge 0$. So $\we(A\weg\set{\alpha x:\alpha>0})$ does not contain $x$. So for some $\alpha>0$, $\alpha x\in A$. 
\end{proof}

\begin{lemma}\label{lem:aconegeneratingsetofextremepointsisminimal}
Let $K$ be a cone. Let $A$ be a set of all extreme elements, so that for all $x,y\in A, x\neq y$, we have that $x\neq \alpha y$, for all $\alpha>0$. If $K$ is generated by $A$, then $A$ is minimal. 
\end{lemma}
\begin{proof}
    Suppose $A$ is not minimal, then for some $x\in A$ and for some $x_1,\ldots,x_n\in A\weg\set x$, and $\alpha_1,\ldots,\alpha_n>0$, \[
    x = \sum_{i=1}^n\alpha_ix_i.
    \]
    As $x_i\neq \gamma x$ for all $\gamma>0$ and $i\in\set{1,\ldots,n}$, it follows that $x$ is not extreme. Contradiction. So $A$ is minimal. 
\end{proof}

A corollary to \cref{lem:aconegeneratingsetofextremepointsisminimal,lem:aminimalsetthatgeneratesaconeconsistofonlyextremeelements} is 
\begin{corollary}\label{cor:aconehasaminimalsetifandonlyifitisgeneratedbyitsextremeelements}
    Let $K$ be a cone. Then $K$ has a minimal set if and only if $K$ is generated by its extreme points. 
\end{corollary}

\begin{lemma}\label{lem:uniquenessofminimalsetthatgeneratescone}
    Let $K$ be a cone that is generated by minimal sets $A$ and $B$, so $K=\cone(A)=\cone(B)$. Then for every $x\in A$ there is an $\alpha>0$ so that $\alpha x\in B$, and vice versa, for every $y\in B$ there is a $\beta>0$ so that $\beta y \in A$. 
\end{lemma}
\begin{proof}
    A minimal set consists of extrema of $K$, which are unique up to a constant. 
\end{proof}

\begin{lemma}\label{lem:aconegeneratedbyAwithuniquedecompositionthenAisminimalandconsistofextremepoints}
    If every element in $\cone(A)$ has a unique decomposition in terms of $A$, then $A$ is a minimal set, and the elements of $A$ are extrema of $\cone(A)$.   
\end{lemma}
\begin{proof}
    As every element has a unique decomposition, for every $x\in A$, $x\notin \cone(A\weg\set x)$. So $A$ is minimal. It follows from \cref{lem:aminimalsetthatgeneratesaconeconsistofonlyextremeelements} that $A$ is a set of extreme elements. 
\end{proof}

\end{document}